\documentclass[final,notitlepage,12pt,reqno]{amsart}
\usepackage{graphicx}
\usepackage{calc}
\usepackage{url}
\usepackage{mathrsfs}
\newlength{\depthofsumsign}
\makeatletter
\let\I\@undefined
\makeatother
\usepackage{geometry}
\usepackage{indentfirst}
\usepackage[normalem]{ulem}
\usepackage{float}
\usepackage{amsthm}

\usepackage{enumitem}[2011/09/28]
\setenumerate{align=left, leftmargin=0pt,labelsep=.5em, labelindent=0\parindent,listparindent=\parindent,itemindent=*}
\usepackage{array}
\usepackage{empheq}
\usepackage{natbib}
\usepackage[all]{xy}

\usepackage{multirow}

\usepackage{caption}[2012/02/19]

\usepackage{longtable,lscape}

\usepackage{fouriernc}
\usepackage{fourier}
\usepackage{amssymb,bm,amsmath}
\usepackage{amsfonts}

\usepackage[OT2,T1,T2A]{fontenc}

\usepackage[english,french,german,russian]{babel}
\usepackage{appendix}

\newcommand{\textcyr}[1]{
{\fontencoding{OT2}\fontfamily{wncyr}\selectfont
\hyphenchar\font=-1\relax#1}}

\DeclareMathOperator{\D}{d}
\DeclareMathOperator{\I}{Im}
\DeclareMathOperator{\R}{Re}

\def\XXint#1#2#3{{\setbox0=\hbox{$#1{#2#3}{\int}$}
     \vcenter{\hbox{$#2#3$}}\kern-.5\wd0}}

\bibpunct{[}{]}{,}{n}{}{;}

\def\eor{\hfill$ \square$}

\theoremstyle{plain}
\newtheorem{theorem}{Theorem}[section]
\newtheorem{proposition}[theorem]{Proposition}
\newtheorem{lemma}[theorem]{Lemma}

\newenvironment{remark}[1][Remark]{\begin{trivlist}
\item[\hskip \labelsep {\bfseries #1}]}{\end{trivlist}}

\theoremstyle{definition}

\numberwithin{equation}{section}

\usepackage{color}

\begin{document}

\pagenumbering{roman}
\selectlanguage{english}
\title[Two Definite Integrals Involving Products of Four Legendre Functions]{Two Definite Integrals\\ Involving Products of Four Legendre Functions}
\author{Yajun Zhou
}
\address{
Program in Applied and Computational Mathematics, Princeton University, Princeton, NJ 08544, USA
}
\email{\vspace{2em}yajunz@math.princeton.edu}
\date{\today}

\maketitle

\begin{abstract}
    The definite integrals $ \int_{-1}^1x[P_\nu(x)]^4\D x$ and $ \int_{0}^1x[P_\nu(x)]^2\{[P_\nu(x)]^2-[P_\nu(-x)]^2\}\D x$ are evaluated in closed form, where $ P_\nu$ stands for the Legendre function of degree $ \nu\in\mathbb C$. Special cases of these integral formulae have appeared in arithmetic studies of automorphic Green's functions and Epstein zeta functions. \textit{Subject Classification (AMS 2010)}: 33C05, 33C15\ (Primary), 11F03 (Secondary) \end{abstract}


\pagenumbering{arabic}
\section{Introduction}

In \cite[][\S6, Eq.~77]{Zhou2013Pnu} and \cite[][Remark 3.3.4.1, Eq.~3.3.38]{AGF_PartII}, we stated the following definite integral formulae \begin{align}
\int_{-1}^1
x[P_\nu(x)]^4\D x={}&\frac{2\sin^4(\nu\pi )[\psi^{(2)}(\nu+1)+\psi^{(2)}(-\nu)+28\zeta(3)]}{(2\nu+1)^2\pi^4}\label{eq:Int4Pnu}\intertext{and}\int_{0}^1x[P_\nu(x)]^2\{[P_\nu(x)]^2-[P_\nu(-x)]^2\}\D x={}&\frac{4\sin^2(\nu\pi)}{(2\nu+1)^{2}\pi^{2}}\left[ \frac{\psi^{(0)}(\nu+1)+\psi^{(0)}(-\nu)}{2}+ \gamma_0+2\log2\right]\label{eq:Int4Pnu_mir}
\end{align}without elaborating on their proof. Here, the Legendre function of the first kind is defined  by the Mehler--Dirichlet integral  \cite[][p.~22, Eq.~1.6.28]{ET1} \begin{align}
P_\nu(\cos\theta)=\frac{2}{\pi}\int_0^\theta \frac{\cos\frac{(2\nu+1)\beta}{2}}{\sqrt{\smash[b]{2(\cos\beta-\cos\theta)}}}\D\beta,\quad  \theta\in(0,\pi),\nu\in\mathbb C;
\end{align} the polygamma functions $\psi^{(m)}(z)=\D^{m+1}\log\Gamma(z)/\D z^{m+1},m\in\mathbb Z_{\geq0} $ are logarithmic derivatives of the  Euler gamma function; the Euler--Mascheroni constant $ \gamma_0:=\lim_{n\to\infty}\left(-\log n +\sum_{k=1}^n\frac1k\right)$ is given by  $ -\psi^{(0)}(1/2)-2\log 2$ and Ap\'ery's constant $ \zeta(3)= \sum_{n=1}^\infty n^{-3}$ is  $ -\psi^{(2)}(1/2)/14$. On the right-hand sides of \eqref{eq:Int4Pnu} and \eqref{eq:Int4Pnu_mir}, the expressions are well-defined for $ \nu\in\mathbb C\smallsetminus(\mathbb Z\cup\{-1/2\})$, and  extend to all $\nu\in\mathbb C$, by continuity.

Some special cases of \eqref{eq:Int4Pnu} and \eqref{eq:Int4Pnu_mir} have shown up in arithmetic studies of automorphic Green's functions (see \cite[][Eqs.~2.3.18 and 2.3.40]{AGF_PartI} as well as  \cite[][Remark 3.3.4.1, Eqs.~3.3.34--3.3.37]{AGF_PartII}) and Epstein zeta functions \cite[][Eqs.~4.18 and 4.28]{EZF}, namely,\begin{align}
\zeta(3)={}&-\frac{2\pi^{4}}{189}\int_{-1}^1 x[P_{-1/6}(x)]^4\D x=-\frac{\pi^{4}}{168}\int_{-1}^1 x[P_{-1/4}(x)]^4\D x=-\frac{\pi^{4}}{243}\int_{-1}^1 x[P_{-1/3}(x)]^4\D x,\\\zeta(5)={}&-\frac{\pi^4}{372}\int_{-1}^1
x[P_{-1/2}(x)]^4\D x,\label{eq:zeta5_int_repn}
\end{align}and\begin{align}
\frac{8\pi^2}{9}\int_0^1x[P_{-1/6}(x)]^2\{[P_{-1/6}(x)]^{2}-[P_{-1/6}(-x)]^{2}\}\D x={}&-3\log3,\label{eq:G3_Hecke1_spec}\\\frac{\pi^2}{8}\int_0^1x[P_{-1/4}(x)]^2\{[P_{-1/4}(x)]^{2}-[P_{-1/4}(-x)]^{2}\}\D x={}&-\log2,\label{eq:G3_Hecke2_spec}\\\frac{\pi^2}{27}\int_0^1x[P_{-1/3}(x)]^2\{[P_{-1/3}(x)]^{2}-[P_{-1/3}(-x)]^{2}\}\D x={}&2\log2-\frac{3}{2}\log3,\label{eq:G3_Hecke3_spec}\\\frac{\pi^2}{7}\int_0^1x[P_{-1/2}(x)]^2\{[P_{-1/2}(x)]^{2}-[P_{-1/2}(-x)]^{2}\}\D x={}&-\zeta(3).
\end{align}Here, $ \zeta(5)=\sum_{n=1}^\infty n^{-5}$.
These special cases have already been verified by independent methods in the aforementioned references.
For example, \eqref{eq:G3_Hecke1_spec}--\eqref{eq:G3_Hecke3_spec} evaluate, respectively, the following  weight-6 automorphic Green's functions \cite[cf.][Eqs. 2.3.18--2.3.20]{AGF_PartI}: \begin{align}
G_{3}^{\mathfrak H/\overline{\varGamma}_0(1)}\left( \frac{1+i\sqrt{3}}{2},i \right),\quad G_{3}^{\mathfrak H/\overline\varGamma_0(2)}\left( \frac{i-1}{2},\frac{i}{\sqrt{2}}\right)\quad \text{and}\quad G_{3}^{\mathfrak H/\overline\varGamma_0(3)}\left( \frac{3+i\sqrt{3}}{6},\frac{i}{\sqrt{3}}\right),
\end{align} where \begin{align}
G_{3}^{\mathfrak H/\overline\varGamma_0(N)}(z_{1},z_{2}):=-\sum_{\substack{a,b,c,d\in\mathbb Z\\ ad-Nbc=1}}Q_{2}
\left( 1+\frac{\left\vert z_{1} -\frac{a z_2+b}{Ncz_{2}+d}\right\vert ^{2}}{2\I z_1\I\frac{a z_2+b}{Ncz_{2}+d}} \right),
\end{align}  with $Q_2(t)=-\frac{3t}{2}+\frac{3t^2-1}{4}\log\frac{t+1}{t-1}$ for $t>1$.
One can compute these special values of automorphic Green's functions  by other means (see \cite[][Chap.~IV, Proposition 2]{GrossZagierII} and \cite[][Remark 3.3.4.1]{AGF_PartII}).
 In addition to descending from the symmetry  of Epstein zeta functions \cite[][Eq.~4.18]{EZF}, the evaluation in \eqref{eq:zeta5_int_repn} has also appeared in the studies of lattice sums by  Wan and Zucker \cite[][Eq.~42]{WanZucker2014}.

In this article, we establish \eqref{eq:Int4Pnu} and \eqref{eq:Int4Pnu_mir} in full generality, drawing on an analytic technique explored extensively in \cite{Zhou2013Pnu,Zhou2013Int3Pnu}: the Hansen--Heine scaling limits that relate Legendre functions $P_\nu$ of large degrees $ |\nu|\to\infty$ to Bessel functions (see \cite[][\S5.7 and \S5.71]{Watson1944Bessel} as well as \eqref{eq:Pnu_J0} and \eqref{eq:Qnu_Y0} in this article). The article is organized as follows. In \S\ref{sec:Int4Bessel}, we present a collection of integrals, old and new, over products of four Bessel functions. In \S\ref{sec:main_ids}, we  investigate the asymptotic behavior of the definite integrals in  \eqref{eq:Int4Pnu} and \eqref{eq:Int4Pnu_mir}, based on  the analysis in \S\ref{sec:Int4Bessel}, and verify the integral formulae for all $\nu\in\mathbb C$ by contour integrations.

 \section{Some definite integrals involving products of four Bessel functions\label{sec:Int4Bessel}}

 We  use $ P_\nu$ to define the  Legendre  functions of the second kind, as follows: \begin{align}\label{eq:def_Qnu}Q_\nu(x):=\frac{\pi}{2\sin(\nu\pi)}[\cos(\nu\pi)P_\nu(x)-P_{\nu}(-x)],\quad \nu\in\mathbb C\smallsetminus\mathbb Z;\quad Q_n(x):=\lim_{\nu\to n}Q_\nu(x),\quad n\in\mathbb Z_{\geq0}\end{align}for $ -1<x<1$. For $ \nu\in\mathbb C,-\pi<\arg z\leq\pi$,  the Bessel functions $ J_\nu$ and $Y_\nu$ are defined by\begin{align}
J_\nu(z):={}&\sum_{k=0}^\infty\frac{(-1)^k}{k!\Gamma(k+\nu+1)}\left( \frac{z}{2} \right)^{2k+\nu},&Y_\nu(z):={}&\lim_{\mu\to\nu}\frac{J_\mu(z)\cos(\mu\pi )-J_{-\mu}(z)}{\sin(\mu\pi)},
\intertext{in parallel to the modified Bessel functions
$I_\nu$ and $K_\nu$:}I_\nu(z):={}&\sum_{k=0}^\infty\frac{1}{k!\Gamma(k+\nu+1)}\left( \frac{z}{2} \right)^{2k+\nu},&K_\nu(z):={}&\frac{\pi}{2}\lim_{\mu\to\nu}\frac{I_{-\mu}(z)-I_\mu(z)}{\sin(\mu\pi)}.\end{align}

As in  \cite[][\S3]{Zhou2013Pnu}, our arguments in this paper draw heavily on the following asymptotic formulae that connect Legendre functions   $P_\nu$ and $Q_\nu$ to Bessel functions $ J_0$ and $Y_0$ (see  \cite[][Eqs.~43 and 46]{Jones2001} or \cite[][Chap.~12, Eqs.~12.18 and 12.25]{Olver1974}):\begin{align}P_{\nu}(\cos\theta)={}&\sqrt{\frac{\theta}{\sin\theta}}J_0\left( \frac{(2\nu+1)\theta}{2}\right)+O\left( \frac{1}{2\nu+1}  \right),\label{eq:Pnu_J0}\\ Q_\nu(\cos\theta)={}&-\frac{\pi}{2}\sqrt{\frac{\theta}{\sin\theta}}Y_0\left(\frac{(2\nu+1)\theta}{2}\right)+O\left( \frac{1}{2\nu+1}  \right),\label{eq:Qnu_Y0}\end{align}where the error bounds are uniform for  $ \theta\in(0,\pi/2]$, so long as $ |\nu|\to\infty,-\pi<\arg \nu<\pi$ \cite{Jones2001,Olver1974}.
\begin{lemma}We have the following evaluations:\begin{align}
\int_0^\infty xJ_0(x)[Y_0(x)]^3\D x =\int_0^\infty x[J_0(x)]^{3}Y_0(x)\D x=-\frac{1}{4\pi},\label{eq:JY3_J3Y}
\end{align}and a vanishing identity:\begin{align}
\int_0^\infty x[J_0(x)]^{2}\left\{[J_0(x)]^{2}-3[Y_0(x)]^2\right\}\D x=0,\label{eq:J4_3J2Y2}
\end{align}where it is understood that  $ \int_0^\infty f(x)\D x:=\lim_{M\to+\infty}\int_0^M f(x)\D x$.\end{lemma}\begin{proof}We first point out that the Hankel function  $ H_0^{(1)}(z)=J_0(z)+i Y_0(z)$ satisfies \begin{align}
\mathscr P\int_{-\infty+i0^+}^{+\infty}z[H_0^{(1)}(z)]^4\D z=0,\label{eq:HankelH1_4th_power_int}
\end{align}where the Cauchy principal value is taken: $ \mathscr P\int^{+\infty}_{-\infty+i0^+}f(z)\D z:=\lim_{M\to+\infty}\left(\int^{i0^+}_{-M+i0^+}+\int_{i0^+}^M\right)f(z)\D z$. The formula above  follows directly from the asymptotic behavior \cite[][\S7.2, Eq.~1]{Watson1944Bessel}\begin{align}
H_0^{(1)}(z)=\sqrt{\frac{2}{\pi z}}e^{i\left(z-\frac{\pi}{4}\right)}\left[ 1+O\left(\frac{1}{|z|}\right) \right],\quad |z|\to+\infty,-\pi<\arg z<2\pi,\label{eq:HankelH1_0_asympt}
\end{align}and an application of Jordan's lemma to a semi-circle
in the
upper half-plane. Noting that $ H_0^{(1)}(-x+i0^+)=-J_{0}(x)+i Y_0(x)$ for $ x>0$, we can convert \eqref{eq:HankelH1_4th_power_int} into \begin{align}
8i\left\{\int_0^\infty x[J_0(x)]^{3}Y_0(x)\D x-\int_0^\infty xJ_0(x)[Y_0(x)]^3\D x\right\} =0,
\end{align}which implies the first equality in \eqref{eq:JY3_J3Y}.

We  now use asymptotic analysis and residue calculus to establish a formula\begin{align}
\mathscr P\int_{-\infty+i0^+}^{+\infty}\left\{z[J_0(z)]^{2}[H_0^{(1)}(z)]^2-\frac{1-e^{4iz}}{\pi^{2}z}\right\}\D z=0,\label{eq:J2H12}
\end{align}and read off its imaginary part as\begin{align}
4\int_0^\infty x[J_0(x)]^{3}Y_0(x)\D x+\frac{1}{\pi^{2}}\int_{-\infty}^\infty\frac{\sin(4x)\D x}{x}=0.
\end{align}  Thus we arrive at the evaluation\begin{align}
\int_0^\infty x[J_0(x)]^{3}Y_0(x)\D x=-\frac{1}{4\pi},\label{eq:J3Y}
\end{align}upon invoking the familiar Dirichlet integral $\int_{-\infty}^\infty\frac{\sin x}{x}\D x=\pi$.

To prove \eqref{eq:J4_3J2Y2}, simply rewrite \begin{align}
\mathscr P\int_{-\infty+i0^+}^{+\infty}z J_{0}(z)[H_0^{(1)}(z)]^3\D z=0
\end{align} with the knowledge that $ [H_0^{(1)}(x)]^3-[H_0^{(1)}(-x+i0^+)]^3=[J_0(x)+i Y_0(x)]^{3}-[-J_{0}(x)+i Y_0(x)]^{3}=2[J_0(x)]^3-6J_{0}(x)[Y_0(x)]^2$.\end{proof}\begin{remark}We note that the following integrals for $ \R\nu>-1/2$\begin{subequations}\begin{align}\int_0^\infty x[J_\nu(bx)]^{2}J_\nu(cx)Y_\nu(cx)\D x={}&0,& 0<{}&b<c\intertext{and}\int_0^\infty x[J_\nu(bx)]^{2}J_\nu(cx)Y_\nu(cx)\D x={}&-\frac{1}{2\pi bc},& 0<{}&c<b\end{align} \end{subequations} have been tabulated in \cite[][item 2.13.24.2]{PBMVol2}. Formally, our \eqref{eq:J3Y}  is an average of the two integral formulae displayed above. \eor\end{remark}
\begin{lemma}We have the following evaluation:\begin{align}
\int_0^\infty x\left\{[J_0(x)]^{4}-6[J_0(x)Y_0(x)]^2+[Y_0(x)]^{4}\right\}\D x=-\frac{14\zeta(3)}{\pi^{4}} .\label{eq:J4_6J2Y2_Y4}
\end{align}\end{lemma}\begin{proof}We recast the  stated integral  into   \begin{align}
\R\int_0^\infty z[H_0^{(1)}(z)]^4\D z=-\frac{14\zeta(3)}{\pi^{4}}.
\end{align}To show the identity above, deform the contour to the positive $\I z$-axis, use the relation $ H_0^{(1)}(iy)=2iK_{0}(y)/\pi,\forall y>0$, along with the following formula:\begin{align}
\int_0^\infty t[K_0(t)]^4\D t=\frac{7\zeta(3)}{8},
\end{align} which is found in \cite[][p.~6]{Ouvry2005}, \cite[][Eq.~202]{Groote2007} and \cite[][p.~19, Table 1]{Bailey2008}. (Thanks to  the extensive and systematic research on  ``Bessel moments'' in the references just mentioned, it is now possible to evaluate $ \int_0^\infty t^{\ell+mn}[K_m(t)]^n\D t$ for $ \ell,m\in\mathbb Z_{\geq0},n\in\{1,2,3,4\}$ in closed form, as implemented in symbolic computation software like \textit{Mathematica}. In contrast, only sporadic results are known when five or more Bessel factors are involved \cite[][\S5 and \S6]{Bailey2008}.)\end{proof}

\begin{lemma}We have the following integral identity:\begin{align}\int_0^\infty \left(x[J_0(x)]^{2}\{[Y_0(x)]^2-[J_0(x)]^{2}\}+\frac{1-\cos(4x)}{\pi^{2}x}\right)\D x
=0.\label{eq:J2Y2J2_cos}
\end{align}\end{lemma}\begin{proof}The claimed formula is equivalent to the statement that \begin{align}
\R\int_0^\infty \left\{\frac{1-e^{4iz}}{\pi^{2}z}-z[J_0(z)]^{2}[H_0^{(1)}(z)]^2\right\}\D z=0.
\end{align}    Deforming the contour of integration to the positive $\I z$-axis, we see that it suffices to verify another vanishing identity: \begin{align}
\int_0^\infty\left\{ \frac{1-e^{-4y}}{y}-4y[I_0(y)]^2[K_0(y)]^2 \right\}\D y=0.
\end{align}To put the formula above in broader context, we will demonstrate that \begin{align}\int_0^\infty y\left\{ I_{\nu-1/2}(y)I_{\nu+1/2}(y)K_{\nu-1/2}(y)K_{\nu+1/2}(y)-[I_\nu(y)]^2[K_\nu(y)]^2 \right\}\D y=0\label{eq:IIKK-IIKK}\end{align}holds as long as $\R\nu>-1/2$. We start from the following identity \cite[cf.][\S13.6, Eq.~3]{Watson1944Bessel}:\begin{align}
I_\nu(y)K_\nu(y)={}&\int_0^{\pi/2}\frac{J_{2\nu}(2y\tan\phi)\D\phi}{\cos\phi},& \R\nu>{}&-\frac{1}{2},
\intertext{and transform it with the aid of a standard recurrence formula  $ 2\partial J_\nu(z)/\partial z=J_{\nu-1}(z)- J_{\nu+1}(z)$ for Bessel functions \cite[][\S3.2, Eq.~2]{Watson1944Bessel}. In particular, we have}
I_{\nu-1}(y)K_{\nu-1}(y)-I_\nu(y)K_\nu(y)={}&\int_0^{\pi/2}\frac{J_{2\nu-1}(2y\tan\theta)\sin\theta\D\theta}{y},& \R\nu>{}&0,\\I_{\nu-1/2}(y)K_{\nu-1/2}(y)-I_{\nu+1/2}(y)K_{\nu+1/2}(y)={}&\int_0^{\pi/2}\frac{J_{2\nu}(2y\tan\theta)\sin\theta\D\theta}{y},& \R\nu>{}&-\frac{1}{2},
\end{align} after integration by parts. If we now denote the left-hand side of \eqref{eq:IIKK-IIKK} by  $f(\nu)$, then \begin{align}
f(\nu)-f(\nu-1/2)={}&\int_0^\infty\D y\iint_{0<\theta<\pi/2,0<\phi<\pi/2}\D\theta\D\phi\,\frac{\sin\theta}{\cos\phi}\times\notag\\&\times [J_{2\nu-1}(2y\tan\theta)J_{2\nu}(2y\tan\phi)-J_{2\nu}(2y\tan\theta)J_{2\nu-1}(2y\tan\phi)],
\end{align}where the integration over $y$ has a closed form \cite[][\S13.42, Eq.~8]{Watson1944Bessel}:\begin{align}
\int_0^\infty J_\mu(at)J_{\mu-1}(bt)\D   t=\begin{cases}b^{\mu-1}/a^\mu,\ & a>b>0, \\
0, & 0<a<b, \\
\end{cases}
\end{align} and one may indeed interchange  the order of integrations, upon introducing a convergent factor $ e^{-\varepsilon y},\varepsilon\to0^+$. This brings us  \begin{align}
f(\nu)-f(\nu-1/2)=\iint_{0<\theta<\phi<\pi/2}\left( \frac{\tan\theta}{\tan\phi} \right)^{2\nu}\frac{\cos\theta}{2\cos\phi}\D\theta\D\phi-\iint_{0<\phi<\theta<\pi/2}\left( \frac{\tan\phi}{\tan\theta} \right)^{2\nu}\frac{\sin\theta}{2\sin\phi}\D\theta\D\phi.
\end{align}Here, both double integrals are convergent when $\R \nu>0$, and they exactly cancel each other (as is evident from variable substitutions $ \theta\mapsto\frac\pi2-\theta$ and $\phi\mapsto\frac\pi2-\phi$). Thus, we have $ f(\nu)=f(\nu-1/2)$ as long as $\R\nu>0$. Therefore, the relation $ f(\nu)=\lim_{n\to\infty}f(\nu+n/2)$ holds for $\R\nu>-1/2$. When $ \nu>-1/2$, we can compute the limit  $ \lim_{n\to\infty}f(\nu+n/2)=0 $ through the dominated convergence theorem. Finally, we  claim that $ f(\nu)=0,\R\nu>-1/2$ follows from analytic continuation.   \end{proof} \begin{remark}It is worth noting that the ``Bessel moments'' in the form of $ \int_0^\infty t^k[I_0(t)]^2[K_0(t)]^{n-2}\D t$ (where $n\geq5,k\geq0$) appear in the investigations  of two-dimensional quantum field theory \cite[][Eq.~13]{Bailey2008}. More generally, in $D$-dimensional space, the free propagator of a massive particle is expressible in terms of $ K_{(D-2)/2}$  \cite[][Eq.~6]{Groote2007}, and the product of Bessel functions $ I_{(D-2)/2}K_{(D-2)/2}$ arises from angular average of the aforementioned propagator \cite[][Eq.~227]{Groote2007}. Thus, it might be appropriate to ask for some physical interpretations or implications of the cancellation formula in \eqref{eq:IIKK-IIKK}.\eor\end{remark}

\section{Proof of the main identities\label{sec:main_ids}}
In this section, we shall refer to  the left- and right-hand sides of \eqref{eq:Int4Pnu} (resp.~\eqref{eq:Int4Pnu_mir}) as $ A_L(\nu)$ and $A_R(\nu)$ (resp.~$ B_L(\nu)$ and $B_R(\nu)$).
It is clear that all these four functions are invariant under the variable substitution $ \nu\mapsto-\nu-1$. Furthermore, one can show that the Taylor expansions\begin{align}
(2\nu +1)^{2}A_R(\nu)={}&4(\nu-n)-\frac{8\pi^{3}}{3}(\nu-n)^3+O((\nu-n)^4),\label{eq:AR_Taylor}\\(2\nu +1)^{2}B_R(\nu)={}&2(\nu-n)+O((\nu-n)^2)\label{eq:BR_Taylor}
\end{align}hold around every non-negative integer $ n$. Before achieving our final goal of proving $A_L(\nu)=A_R(\nu)$ and $ B_L(\nu)=B_R(\nu)$, we need to check  that the Taylor expansions of $ (2\nu +1)^2A_L(\nu)$ and $ (2\nu+1)^2 B_L(\nu)$ agree with their respective counterparts in \eqref{eq:AR_Taylor} and \eqref{eq:BR_Taylor}.

\begin{lemma}\label{lm:AL_BL_Taylor}\begin{enumerate}[leftmargin=*,  label=\emph{(\alph*)},ref=(\alph*),
widest=a, align=left]\item We have an integral formula for $\nu\in\mathbb C$:\begin{align}
(2\nu+1)^{2}\int_{-1}^1x[P_\nu(x)]^3P_\nu(-x)\D x={}&\frac{\sin(2\nu\pi)\cos(\nu\pi)}{\pi}\label{eq:P3P_eval}\end{align}and a vanishing identity for $n\in\mathbb Z_{\geq0}$:\begin{align}
\int_{-1}^1xP_n(x)Q_{n}(x)\left\{ \frac{4}{\pi^{2}} [Q_{n}(x)]^{2}-[P_n(x)]^2\right\}\D x=0.\label{eq:PQQQ0}
\end{align}\item We have the following  Taylor expansions\begin{align}
(2\nu +1)^{2}A_L(\nu)={}&4(\nu-n)-\frac{8\pi^{2}}{3}(\nu-n)^3+O((\nu-n)^4),\label{eq:AL_Taylor}\\(2\nu +1)^{2}B_L(\nu)={}&2(\nu-n)+O((\nu-n)^2)\label{eq:BL_Taylor}
\end{align} as $\nu$ approaches $ n\in\mathbb Z_{\geq0}$. \end{enumerate}\end{lemma}\begin{proof}\begin{enumerate}[leftmargin=*,  label=(\alph*),ref=(\alph*),
widest=a, align=left]\item The proof of \eqref{eq:P3P_eval} is found in \cite[][Corollary 5.3]{Zhou2013Pnu}. It was based on the Tricomi pairing (see \cite[][\S4.3, Eq.~2]{TricomiInt} or \cite[][Eq.~11.237]{KingVol1})\begin{align}
\int_{-1}^1f(x)\left[ \mathscr P \int_{-1}^1\frac{g(\xi)\D\xi}{\pi (x-\xi)} \right]\D x+\int_{-1}^1g(x)\left[ \mathscr P\int_{-1}^1\frac{f(\xi)\D\xi}{\pi (x-\xi)} \right]\D x=0\label{eq:Tricomi_pairing}
\end{align} of the following formulae \cite[][Proposition 5.2 and Corollary 5.3]{Zhou2013Pnu}\begin{align}
\frac{2\sin(\nu\pi)}{\pi}\mathscr P\int_{-1}^1\frac{P_\nu(\xi)P_\nu(-\xi)\D \xi}{x-\xi}={}&[P_{\nu}(x)]^{2}-[P_\nu(-x)]^2,\label{eq:Tricomi_PP}\\\frac{2\sin(\nu\pi)}{\pi}\mathscr P\int_{-1}^1\frac{\xi P_\nu(\xi)P_\nu(-\xi)\D \xi}{x-\xi}={}&x\{[P_{\nu}(x)]^{2}-[P_\nu(-x)]^2\}-\frac{2\sin(2\nu\pi)}{(2\nu+1)\pi},\label{eq:Tricomi_xPP}
\end{align} which are valid for $  x\in(-1,1),\nu\in\mathbb C$, along with an integral evaluation \cite[][Eq.~19$_{(0,\nu)}$]{Zhou2013Pnu}\begin{align}
\int_{-1}^1P_\nu(\xi)P_\nu(-\xi)\D \xi=\frac{2\cos(\nu\pi)}{2\nu+1},\quad \nu\in\mathbb C\smallsetminus\{-1/2\}.\label{eq:int_PP}
\end{align}

To justify \eqref{eq:PQQQ0}, we need the Tricomi pairing of the following formulae  for $n\in\mathbb Z_{\geq0}$:\begin{align}\frac4{\pi^2}\mathscr
P\int_{-1}^1\frac{P_n(\xi)Q_n(\xi)\D \xi}{x-\xi}={}&\frac{4}{\pi^{2}} [Q_{n}(x)]^{2}-[P_n(x)]^2,\label{eq:PQ_T}\\\frac4{\pi^2}\mathscr
P\int_{-1}^1\frac{\xi P_n(\xi)Q_n(\xi)\D \xi}{x-\xi}={}&x\left\{\frac{4}{\pi^{2}} [Q_{n}(x)]^{2}-[P_n(x)]^2\right\}.\label{eq:xPQ_T}
\end{align}
To prove  \eqref{eq:PQ_T}, apply the Hardy--Poincar\'e--Bertrand formula (see \cite[][\S4.3, Eq.~4]{TricomiInt} or \cite[][Eq.~11.52]{KingVol1}) to the Neumann integral representation of  $ Q_n$ \cite[][Eq.~11.269]{KingVol1}:\begin{align}
Q_{n}(x)=\mathscr P\int_{-1}^1\frac{P_n(\xi)\D \xi}{2(x-\xi)},\quad \forall x\in(-1,1),\forall n\in\mathbb Z_{\geq0.}
\end{align}To deduce \eqref{eq:xPQ_T} from \eqref{eq:PQ_T},  use a vanishing identity $ \int_{-1}^1P_n(x)Q_n(x)\D x=0,\forall n\in\mathbb Z_{\geq0}$ \cite[][item 2.18.13.3]{PBMVol3}.

              \item The symmetry of Legendre polynomials $ P_n(-x)=(-1)^nP_n(x),n\in\mathbb Z_{\geq0}$ gives $(2n +1)^{2}A_L(n)=(2n +1)^{2}B_L(n)=0 $.
So it remains to check the derivatives of $ A_L(\nu)$ and $B_L(\nu)$ at $\nu=n\in\mathbb Z_{\geq0}$.
 In what follows, we define  \begin{align}
P_n^{\{m\}}(x):=\left.\frac{\partial^m}{\partial\nu^m}\right|_{\nu=n}P_\nu(x),\quad m\in\mathbb Z_{>0},
\end{align}to simplify notations.

 For the linear order Taylor expansions of $ (2\nu +1)^{2}A_L(\nu)$ and $ (2\nu+1)^2B_L(\nu)$, we compute\begin{align}
(2n +1)^{2}A'_L(n)={}&4(2n +1)^{2}\int_{-1}^1x[P_n^{\vphantom{\{0\}}}(x)]^3P_n^{\{1\}}(x)\D x\intertext{and}(2n +1)^{2}B'_L(n)={}&2(2n +1)^{2}\int_{0}^1x[P_n^{\vphantom{\{0\}}}(x)]^2\{P_n^{\vphantom{\{0\}}}(x)P_n^{\{1\}}(x)-P_n^{\vphantom{\{0\}}}(-x)P_n^{\{1\}}(-x)\}\D x\notag\\={}&2(2n +1)^{2}\int_{-1}^1x[P_n^{\vphantom{\{0\}}}(x)]^3P_n^{\{1\}}(x)\D  x.
\end{align}Differentiating \eqref{eq:P3P_eval} at $ \nu=n\in\mathbb Z_{\geq0}$, we obtain\begin{align}&3(2n +1)^{2}\int_{-1}^1x[P_n^{\vphantom{\{0\}}}(x)]^2P_{n}^{\vphantom{\{0\}}}(-x)P_n^{\{1\}}(x)\D x+(2n +1)^{2}\int_{-1}^1x[P_n^{\vphantom{\{0\}}}(x)]^3P_n^{\{1\}}(-x)\D x\\={}&
2(2n +1)^{2}(-1)^{n}\int_{-1}^1x[P_n^{\vphantom{\{0\}}}(x)]^3P_n^{\{1\}}(x)\D x=\left.\frac{\partial}{\partial\nu}\right|_{\nu=n}\frac{\sin(2\nu\pi)\cos(\nu\pi)}{\pi}=2(-1)^{n},
\end{align}which yields $ (2n +1)^{2}A'_L(n)=2(2n +1)^{2}B'_L(n)=4$ for all $n\in\mathbb Z_{\geq0} $.

 For the quadratic order expansion of $ (2\nu +1)^{2}A_L(\nu)$,  we compute   \begin{align}&
(2n +1)^{2}A''_L(n)+8(2n+1)A'_{L}(n)\notag\\={}&4(2n +1)^{2}\int_{-1}^1x[P_n^{\vphantom{\{0\}}}(x)]^3P_n^{\{2\}}(x)\D x+12(2n +1)^{2}\int_{-1}^1x[P_n^{\vphantom{\{0\}}}(x)]^2[P_n^{\{1\}}(x)]^{2}\D x+\frac{32}{2n +1},
\end{align}and compare it to the second-order derivative of \eqref{eq:P3P_eval}:\begin{align}&
6(2n +1)^{2}\left\{\int_{-1}^1xP_n^{\vphantom{\{0\}}}(x)P_{n}^{\vphantom{\{0\}}}(-x)[P_n^{\{1\}}(x)]^{2}\D x+\int_{-1}^1x[P_n^{\vphantom{\{0\}}}(x)]^{2}P_n^{\{1\}}(x)P_n^{\{1\}}(-x)\D x\right\}\notag\\{}&+(2n +1)^{2}\left\{3\int_{-1}^1x[P_n^{\vphantom{\{0\}}}(x)]^2P_{n}^{\vphantom{\{0\}}}(-x)P_n^{\{2\}}(x)\D x+\int_{-1}^1x[P_n^{\vphantom{\{0\}}}(x)]^3P_n^{\{2\}}(-x)\D x\right\}+\frac{16(-1)^{n}}{2n+1}\notag\\={}&(-1)^{n}\left\{2(2n +1)^{2}\int_{-1}^1x[P_n^{\vphantom{\{0\}}}(x)]^3P_n^{\{2\}}(x)\D x+6(2n +1)^{2}\int_{-1}^1x[P_n^{\vphantom{\{0\}}}(x)]^2[P_n^{\{1\}}(x)]^{2}\D x+\frac{16}{2n+1}\right\}\notag\\={}&\left.\frac{\partial^{2}}{\partial\nu^{2}}\right|_{\nu=n}\frac{\sin(2\nu\pi)\cos(\nu\pi)}{\pi}=0.
\end{align}Here, the integral \begin{align}
\int_{-1}^1x[P_n^{\vphantom{\{0\}}}(x)]^{2}P_n^{\{1\}}(x)P_n^{\{1\}}(-x)\D  x
\end{align}vanishes because the integrand is an odd function. This shows that  $ (2\nu +1)^{2}A_L(\nu)=4(\nu-n)+O((\nu-n)^3)$.

To prepare for the computation of\begin{align}&
A'''_L(n)\notag\\={}&24\int_{-1}^1xP_n^{\vphantom{\{0\}}}(x)[P_n^{\{1\}}(x)]^{3}\D x+36\int_{-1}^1x[P_n^{\vphantom{\{0\}}}(x)]^{2}P_n^{\{1\}}(x)P_n^{\{2\}}(x)\D x+4\int_{-1}^1x[P_n^{\vphantom{\{0\}}}(x)]^{3}P_n^{\{3\}}(x)\D x
\end{align} occurring in the cubic order expansion, we  rewrite \eqref{eq:PQQQ0} with $ 2Q_n^{\vphantom{\{0\}}}(x)=P_n^{\{1\}}(x)-(-1)^{n}P_n^{\{1\}}(-x)$ (a consequence of \eqref{eq:def_Qnu}), which leads us to \begin{align}
\int_{-1}^1xP_n(x)[P_n^{\{1\}}(x)-(-1)^{n}P_n^{\{1\}}(-x)]^{3}\D x={}&\pi^2\int_{-1}^1x[P_n(x)]^{3}[P_n^{\{1\}}(x)-(-1)^{n}P_n^{\{1\}}(-x)]\D x\notag\\={}&2\pi^2\int_{-1}^1x[P_n(x)]^{3}P_n^{\{1\}}(x)\D x=\frac{2\pi^2}{(2n+1)^2}.
\end{align}It is then clear that {\allowdisplaybreaks\begin{align}
&\left.\frac{\partial^3}{\partial\nu^3}\right|_{\nu=n}\left\{\int_{-1}^1x[P_\nu(x)]^3P_\nu(-x)\D x-\int_{-1}^1x[P_\nu(-x)]^3P_\nu(x)\D x\right\}+\frac{12(-1)^{n}\pi^2}{(2n+1)^2}\notag\\={}&12(-1)^{n}\int_{-1}^1xP_n^{\vphantom{\{0\}}}(x)\{[P_n^{\{1\}}(x)]^{3}-(-1)^{n}[P_n^{\{1\}}(-x)]^{3}\}\D x\notag\\{}&+18(-1)^{n}\int_{-1}^1x[P_n^{\vphantom{\{0\}}}(x)]^{2}[P_n^{\{1\}}(x)P_n^{\{2\}}(x)-P_n^{\{1\}}(-x)P_n^{\{2\}}(-x)]\D x\notag\\{}&-2(-1)^{n}\int_{-1}^1x[P_n^{\vphantom{\{0\}}}(x)]^{3}[P_n^{\{3\}}(x)-(-1)^{n}P_n^{\{3\}}(-x)]\D x=2(-1)^{n}A'''_{L} (n).
\end{align}}This
gives the formula\begin{align}
A'''_{L}(n)=\frac{288}{(2n+1)^4}-\frac{16\pi^2}{(2n+1)^2}.
\end{align} The Taylor expansion in \eqref{eq:AL_Taylor} can now be verified. \qedhere\end{enumerate}\end{proof}
\begin{proposition}\label{prop:psi2_id}The  integral identity in \eqref{eq:Int4Pnu} holds for all $\nu\in\mathbb C$, where the evaluations of the right-hand side at $ \nu\in\{-1/2\}\cup\mathbb Z$ are interpreted as limits of a continuous function in $\nu$.\end{proposition}\begin{proof} Similar to \cite[][Proposition 3.1]{Zhou2013Pnu}, we need to bound the expression\begin{align}
\mathscr A(\nu):=\frac{(2\nu+1)^{2}}{\sin^4(\nu\pi )}[A_L(\nu)-A_R(\nu)]
\end{align}on a square contour $ C_N$ in the complex $\nu$-plane, with vertices $ \frac{1-4N}{4}-iN,\frac{1+4N}{4}-iN,\frac{1+4N}{4}+iN$ and $ \frac{1-4N}{4}+iN$, where $N\in\mathbb Z_{>0}$. Observe that both $ |\cot(\nu\pi )|$ and $ 1/|\sin(\nu\pi )|$ are bounded on $C_N$, and the uniform bounds do not depend on $N$.

First, we consider the case where  $\nu=\frac{1+4N}{4}$ for a large positive integer $N$. By the asymptotic formulae in \eqref{eq:Pnu_J0} and \eqref{eq:Qnu_Y0}, as well as the integral evaluations in \eqref{eq:JY3_J3Y}, \eqref{eq:J4_3J2Y2} and \eqref{eq:J4_6J2Y2_Y4}, we see that the predominant contribution to  \begin{align}
A_L(\nu)={}&\int_0^1 x\left\{[P_\nu(x)]^4-[P_\nu(-x)]^4\right\}\D x\notag\\={}&\int_0^{\pi /2}\left\{[P_\nu(\cos\theta)]^4-\left[ \cos(\nu\pi)P_\nu(\cos\theta)-\frac{2\sin(\nu\pi)}{\pi} Q_\nu(\cos\theta)\right]^4\right\}\sin\theta  \cos\theta\D \theta
\end{align}
is {\allowdisplaybreaks\begin{align}&
\int_0^\infty \theta\left\{\left[J_0\left(\frac{(2\nu+1)\theta}{2}\right)\right]^4-\left[ \cos(\nu\pi)J_0\left(\frac{(2\nu+1)\theta}{2}\right)+\sin(\nu\pi) Y_0\left(\frac{(2\nu+1)\theta}{2}\right)\right]^4\right\}\D \theta\notag\\={}&-\sin^4(\nu\pi)\int_0^\infty \theta\left\{\left[J_0\left(\frac{(2\nu+1)\theta}{2}\right)\right]^4-6\left[J_0\left(\frac{(2\nu+1)\theta}{2}\right)  Y_0\left(\frac{(2\nu+1)\theta}{2}\right)\right]^2+\left[  Y_0\left(\frac{(2\nu+1)\theta}{2}\right)\right]^4\right\}\D \theta\notag\\{}&+2\sin^2(\nu\pi)\int_0^\infty \theta\left[J_0\left(\frac{(2\nu+1)\theta}{2}\right)\right]^2\left\{\left[J_0\left(\frac{(2\nu+1)\theta}{2}\right)\right]^2-3\left[  Y_0\left(\frac{(2\nu+1)\theta}{2}\right)\right]^2\right\}\D \theta\notag\\{}&-2\sin(\nu\pi)\cos(\nu\pi)\int_0^\infty \theta J_0\left(\frac{(2\nu+1)\theta}{2}\right)Y_{0}\left(\frac{(2\nu+1)\theta}{2}\right)\left\{\left[J_0\left(\frac{(2\nu+1)\theta}{2}\right)\right]^2+\left[  Y_0\left(\frac{(2\nu+1)\theta}{2}\right)\right]^2\right\}\D\theta\notag\\{}&-\frac{\sin(4\nu\pi)}{2}\int_0^\infty \theta J_0\left(\frac{(2\nu+1)\theta}{2}\right)Y_{0}\left(\frac{(2\nu+1)\theta}{2}\right)\left\{\left[J_0\left(\frac{(2\nu+1)\theta}{2}\right)\right]^2-\left[  Y_0\left(\frac{(2\nu+1)\theta}{2}\right)\right]^2\right\}\D\theta\notag\\={}&\left[\frac{14\zeta(3)}{\pi^{4}}\sin^4(\nu\pi)+\frac{\sin(\nu\pi)\cos(\nu\pi)}{\pi}\right]\frac{4}{(2\nu +1)^{2}},
\end{align}}which is compatible with the expansion\begin{align}
\frac{A_R(\nu)}{\sin^4(\nu\pi)}=\left[\frac{14\zeta(3)}{\pi^{4}}+\frac{\cos(\nu\pi)}{\pi\sin^3(\nu\pi)}+O\left( \frac{1}{(2\nu +1)^{2}} \right)\right]\frac{4}{(2\nu +1)^{2}}.
\end{align}

We  break down the error bound for $ \mathscr A(\nu)$ into three parts. \begin{enumerate}[leftmargin=*,  label=(\roman*),ref=(\roman*),
widest=iii, align=left]
\item\label{itm:step1}
By the uniform approximations in the Hansen--Heine scaling limits (\eqref{eq:Pnu_J0} and \eqref{eq:Qnu_Y0}), we have\begin{align}&
\int_0^{\pi/2} \left\{\left[J_0\left(\frac{(2\nu+1)\theta}{2}\right)\right]^4-\left[ \cos(\nu\pi)J_0\left(\frac{(2\nu+1)\theta}{2}\right)+\sin(\nu\pi) Y_0\left(\frac{(2\nu+1)\theta}{2}\right)\right]^4\right\}\frac{\theta^{2}\cot\theta\D \theta}{\sin^4(\nu\pi)}\notag\\={}&\int_0^{\pi /2}\left\{[P_\nu(\cos\theta)]^4-\left[ \cos(\nu\pi)P_\nu(\cos\theta)-\frac{2\sin(\nu\pi)}{\pi} Q_\nu(\cos\theta)\right]^4\right\}\frac{\sin\theta  \cos\theta\D \theta}{\sin^4(\nu\pi)}+O\left( \frac{1}{(2\nu+1)^{5/2}} \right).\label{eq:JY_approx_Pnu}
\end{align}Concretely speaking, the error term is bounded by\begin{align}&
O\left( \frac{1}{2\nu+1} \int_0^{\pi/2}p\left(\frac{(2\nu+1)\theta}{2}\right)\theta^{2}\cot\theta\D \theta\right)\notag\\={}&O\left( \frac{1}{(2\nu+1)^{3}} \int_0^{(2\nu+1)\pi/4}p(x)\frac{2x/(2\nu+1)}{\tan(2x/(2\nu+1))}x\D x\right),\label{eq:p_bound}
\end{align}where $p(x)$ is a bivariate cubic polynomial in $J_0(x)$ and $Y_0(x)$, whose coefficients involve positive integer powers of $ \cot(\nu\pi)$ and $ 1/\sin(\nu\pi)$. For $ x\in(0,1)$, the integrand in \eqref{eq:p_bound} is bounded; for $ x\in[1,(2\nu+1)\pi/4)$, we have an upper bound for $x^{3/2}|p(x)| $. Overall, \eqref{eq:p_bound} is bounded by  $ O((2\nu+1)^{-3}\int_0^{(2\nu+1)\pi/4}x^{-1/2}\D x)=O((2\nu+1)^{-5/2})$, as stated in \eqref{eq:JY_approx_Pnu}.
\item\label{itm:step2} As $ 1-\theta\cot\theta=O((2\nu+1)^{-1})$ for $ \theta\in(0,1/\sqrt{2\nu+1})$, we have an estimate \begin{align}
&\int_0^{1/\sqrt{2\nu+1}} \left\{\left[J_0\left(\frac{(2\nu+1)\theta}{2}\right)\right]^4-\left[ \cos(\nu\pi)J_0\left(\frac{(2\nu+1)\theta}{2}\right)+\sin(\nu\pi) Y_0\left(\frac{(2\nu+1)\theta}{2}\right)\right]^4\right\}\frac{\theta^{2}\cot\theta\D \theta}{\sin^4(\nu\pi)}\notag\\={}&\int_0^{1/\sqrt{2\nu+1}} \theta\left\{\left[J_0\left(\frac{(2\nu+1)\theta}{2}\right)\right]^4-\left[ \cos(\nu\pi)J_0\left(\frac{(2\nu+1)\theta}{2}\right)+\sin(\nu\pi) Y_0\left(\frac{(2\nu+1)\theta}{2}\right)\right]^4\right\}\frac{\D \theta}{\sin^4(\nu\pi)}\notag\\{}&+O\left( \frac{1}{(2\nu+1)^{11/4}} \right).\end{align}The rationale behind this is similar to \ref{itm:step1}: one may dissect the error term\begin{align}
O\left( \frac{1}{2\nu+1} \int_0^{1/\sqrt{2\nu+1}}p\left(\frac{(2\nu+1)\theta}{2}\right)\theta\D \theta\right)=O\left( \frac{1}{(2\nu+1)^{3}} \int_0^{\sqrt{2\nu+1}/2}p(x)x\D x\right)
\end{align} into contributions from the ranges $x\in(0,1) $ and $ x\in[1,\sqrt{2\nu+1}/2)$.
\item\label{itm:step3} Using the asymptotic behavior of Bessel functions for large positive $x$ (cf.~\eqref{eq:HankelH1_0_asympt})\begin{align}
J_{0}(x)=\sqrt{\frac{2}{\pi x}}\cos\left( x-\frac{\pi}{4}\right)\left[ 1+O\left(\frac{1}{x}\right) \right],\quad Y_{0}(x)=\sqrt{\frac{2}{\pi x}}\sin\left( x-\frac{\pi}{4}\right)\left[ 1+O\left(\frac{1}{x}\right) \right],\end{align}along with integration by parts, we are able to  verify that
\begin{align}&
\int_{1/\sqrt{2\nu+1}}^{\pi/2} \left\{\left[J_0\left(\frac{(2\nu+1)\theta}{2}\right)\right]^4-\left[ \cos(\nu\pi)J_0\left(\frac{(2\nu+1)\theta}{2}\right)+\sin(\nu\pi) Y_0\left(\frac{(2\nu+1)\theta}{2}\right)\right]^4\right\}\frac{\theta^{2}\cot\theta\D \theta}{\sin^4(\nu\pi)}\notag\\={}&O\left( \frac{1}{(2\nu+1)^{5/2}} \right)
\end{align}and \begin{align}{}&
\int_{1/\sqrt{2\nu+1}}^{\infty} \theta\left\{\left[J_0\left(\frac{(2\nu+1)\theta}{2}\right)\right]^4-\left[ \cos(\nu\pi)J_0\left(\frac{(2\nu+1)\theta}{2}\right)+\sin(\nu\pi) Y_0\left(\frac{(2\nu+1)\theta}{2}\right)\right]^4\right\}\frac{\D \theta}{\sin^4(\nu\pi)}\notag\\={}&O\left( \frac{1}{(2\nu+1)^{5/2}} \right).
\end{align}\end{enumerate}In sum, we have the following bound
estimate\begin{align}
\mathscr A(\nu)=O\left( \frac{1}{\sqrt{2\nu+1}} \right)\label{eq:scrA_estimate}
\end{align}when $ \nu=\frac{1+4N}{4}$ for a large positive integer $N$.

Then,
for generic $\nu\in C_N$ satisfying $ \R\nu\geq-1/2$, we argue that \eqref{eq:scrA_estimate} remains valid, upon modifying steps \ref{itm:step1}--\ref{itm:step3} by contour deformations. For example, a variation on the derivations  in step \ref{itm:step1}  brings us \begin{align}
O\left( \frac{1}{(2\nu+1)^{3}} \int_0^{|2\nu+1|\pi/4}p(x)\frac{2x/(2\nu+1)}{\tan(2x/(2\nu+1))}x\D x\right)=O\left( \frac{1}{(2\nu+1)^{5/2}} \right),
\end{align}while we can bound the following integral  on a circular arc in the complex $z$-plane\begin{align}
\int_{(2\nu+1)\pi/4}^{|2\nu+1|\pi/4}p(z)\frac{2z/(2\nu+1)}{\tan(2z/(2\nu+1))}z\D  z
\end{align}in the spirit of Jordan's lemma.

By virtue of the reflection symmetry $ \mathscr A(\nu)=\mathscr A(-\nu-1)$, we have thus confirmed the error bound in   \eqref{eq:scrA_estimate}
for all  $\nu\in C_N$.
Furthermore, in view of \eqref{eq:AR_Taylor} and \eqref{eq:AL_Taylor}, we are sure  that $ \mathscr A(\nu)$ is analytic in the region bounded by the contour $ C_N$.

Finally, by Cauchy's integral formula, we see that \begin{align} \mathcal A(\nu)=\frac{1}{2\pi i}\lim_{N\to\infty}\oint_{C_N}\frac{\mathscr A(z)\D z}{z-\nu}\end{align} vanishes identically, hence $ A_L(\nu)=A_R(\nu),\forall\nu\in\mathbb C$.
\end{proof}

\begin{proposition}The  integral identity in \eqref{eq:Int4Pnu_mir} holds  for all $\nu\in\mathbb C$, where the evaluations of the right-hand side at $ \nu\in\{-1/2\}\cup\mathbb Z$ are interpreted as limits of a continuous function in $\nu$.\end{proposition}\begin{proof}We will only show that \begin{align}
\mathscr B(\nu):=\frac{(2\nu+1)^{2}}{\sin^2(\nu\pi )}[B_L(\nu)-B_R(\nu)]=O\left( \frac{1}{\sqrt{2\nu+1}} \right)\label{eq:B_nu_bound}
\end{align}for $ \nu=\frac{1+4N}{4}$, where  $N$ is a large positive integer. The rest of the procedures are essentially similar to those in Proposition~\ref{prop:psi2_id}.

Through direct expansions of the digamma functions, we have \begin{align}
\frac{(2\nu+1)^{2}}{\sin^2(\nu\pi )}B_R(\nu)=\frac{2\cot(\nu\pi)}{\pi}+\frac{4(\gamma_{0}+2\log2+\log\nu)}{\pi^{2}}+O\left( \frac{1}{\nu} \right).
\end{align}In the meantime, we  approximate $ B_L(\nu)$ by \begin{align}
\widetilde B(\nu)={}&\int_0^{(2\nu+1)\pi/4} \left[J_0(x)\right]^{2}\left\{\left[J_0(x)\right]^2-\left[ \cos(\nu\pi)J_0(x)+\sin(\nu\pi) Y_0(x)\right]^2\right\}\frac{2x/(2\nu+1)}{\tan(2x/(2\nu+1))}\frac{4x\D x}{(2\nu+1)^{2}}\notag\\={}&\int_0^{(2\nu+1)\pi/4} \left[J_0(x)\right]^{2}\left\{\left[J_0(x)\right]^2-\left[ Y_0(x)\right]^2-2\cot(\nu\pi)J_0(x) Y_0(x)\right\}\frac{2x/(2\nu+1)}{\tan(2x/(2\nu+1))}\frac{4x\sin^{2}(\nu\pi)\D x}{(2\nu+1)^{2}}
\end{align}and estimate the error bounds. By analogy to the proof of Proposition~\ref{prop:psi2_id}\ref{itm:step1}, we  assert that\begin{align}
\frac{(2\nu+1)^{2}}{\sin^2(\nu\pi )}[B_L(\nu)-\widetilde B(\nu)]=O\left( \frac{1}{\sqrt{2\nu+1}} \right).
\end{align} As in the proof of  Proposition~\ref{prop:psi2_id}\ref{itm:step2}--\ref{itm:step3}, we have \begin{align}
&-8\int_0^{(2\nu+1)\pi/4} \left[J_0(x)\right]^{3} Y_0(x)\frac{2x/(2\nu+1)}{\tan(2x/(2\nu+1))}x\D x=\frac{2\cot(\nu\pi)}{\pi}+O\left( \frac{1}{\sqrt{2\nu+1}} \right),
\end{align} according to \eqref{eq:J3Y}, while\begin{align}&
\int_0^{\sqrt{2\nu+1}} \left(x\left[J_0(x)\right]^{2}\left\{\left[ Y_0(x)\right]^2-\left[J_0(x)\right]^2\right\}+\frac{1-\cos(4x)}{\pi^2 x}\right)\frac{2x/(2\nu+1)}{\tan(2x/(2\nu+1))}\D x\notag\\={}&\int_0^{\sqrt{2\nu+1}} \left(x\left[J_0(x)\right]^{2}\left\{\left[ Y_0(x)\right]^2-\left[J_0(x)\right]^2\right\}+\frac{1-\cos(4x)}{\pi^2 x}\right)\frac{x\D x}{(2\nu+1)^{2}}+O\left( \frac{1}{(2\nu+1)^{3/4}} \right)\notag\\={}&-\int_{\sqrt{2\nu+1}}^{\infty} \left(x\left[J_0(x)\right]^{2}\left\{\left[ Y_0(x)\right]^2-\left[J_0(x)\right]^2\right\}+\frac{1-\cos(4x)}{\pi^2 x}\right)\frac{x\D x}{(2\nu+1)^{2}}+O\left( \frac{1}{(2\nu+1)^{3/4}} \right)\notag\\={}&O\left( \frac{1}{\sqrt{2\nu+1}} \right),
\end{align} follows from \eqref{eq:J2Y2J2_cos} and integration by parts on $x\in(\sqrt{2\nu+1},+\infty)$. This subsequently entails\begin{align}
\int_0^{(2\nu+1)\pi/4} \left(x\left[J_0(x)\right]^{2}\left\{\left[ Y_0(x)\right]^2-\left[J_0(x)\right]^2\right\}+\frac{1-\cos(4x)}{\pi^2 x}\right)\frac{2x/(2\nu+1)}{\tan(2x/(2\nu+1))}\D x=O\left( \frac{1}{\sqrt{2\nu+1}} \right).
\end{align}  Next, we  compute a definite integral with the help of integration by parts and a Fourier series expansion \cite[][item 5.4.2.1]{PBMVol1}:{\allowdisplaybreaks\begin{align}&
\int_0^{(2\nu+1)\pi/4} \frac{1-\cos(4x)}{\pi^2 x}\frac{2x/(2\nu+1)}{\tan(2x/(2\nu+1))}\D x=\frac{2}{\pi^{2}}\int_0^{\pi/2}\sin^2 ((2\nu+1)\theta)\cot\theta\D \theta\notag\\={}&-\frac{2(2\nu+1)}{\pi^2}\int_0^{\pi/2}\sin (2(2\nu+1)\theta)\log\sin\theta\D \theta=\frac{2(2\nu+1)}{\pi^2}\int_0^{\pi/2}\sin (2(2\nu+1)\theta)\left[\log2+\sum_{k=1}^\infty \frac{\cos(2k\theta)}{k}\right]\D \theta\notag\\={}&\frac{2\cos^2(\nu\pi)}{\pi^{2}}\log2+\frac{1}{4}\sum_{k=1}^\infty \left(\frac{2}{k}-\frac{1}{k+2 \nu +1}-\frac{1}{k-2 \nu -1}\right) [1+(-1)^k \cos (2  \nu \pi )]\notag\\={}&\frac{3 \nu +1}{2 \pi ^2 \nu\left(2 \nu +1 \right)}+\frac{\psi ^{(0)}(2 \nu )+\gamma _{0}+\log 2}{\pi ^2}+\frac{\cos (2  \nu \pi )}{4 \pi ^2}\left[ \psi ^{(0)}\left(\nu+\frac{1}{2} \right)-2 \psi ^{(0)}(\nu+1 )+\psi ^{(0)}\left(\nu +\frac{3}{2}\right)\right]\notag\\={}&\frac{\gamma_{0}+2\log2+\log\nu}{\pi^{2}}+O\left( \frac{1}{\nu} \right).
\end{align}}This eventually demonstrates that
\begin{align}
\frac{(2\nu+1)^{2}}{\sin^2(\nu\pi )}B_L(\nu)=\frac{2\cot(\nu\pi)}{\pi}+\frac{4(\gamma_{0}+2\log2+\log\nu)}{\pi^{2}}+O\left( \frac{1}{\sqrt{2\nu+1}} \right),
\end{align}so \eqref{eq:B_nu_bound} holds.  \end{proof}

While our foregoing treatments of \eqref{eq:Int4Pnu} and \eqref{eq:Int4Pnu_mir} were motivated by their significance in certain arithmetic problems, the methods just outlined can be equally applied to other integrals with four Legendre factors. In a sense, our previous work \cite{Zhou2013Pnu,Zhou2013Int3Pnu} and this article have provided a framework for evaluating ``Legendre moments'' for (up to) four Legendre factors. This forms an analog for  the researches on ``Bessel moments'' \cite{Ouvry2005,Groote2007,Bailey2008} for (up to) four Bessel factors, which were motivated by Feynman diagrams in quantum field theory.
\subsection*{Acknowledgements} The author is grateful to an anonymous referee for thoughtful comments on polishing the presentation of this paper.

\end{document}